\documentclass{amsart}
\usepackage{amsfonts, amsbsy, amsmath, amssymb}

\ifx\fiverm\undefined 
  \newfont\fiverm{cmr5} 
\fi 

\input prepictex.tex
\input pictex.tex
\input postpictex.tex

\newtheorem{thm}{Theorem}[section]
\newtheorem{lem}[thm]{Lemma}

\newtheorem{conj}[thm]{Conjecture}

\newtheorem{rmk}[thm]{Remark}

\newtheorem{thm-con}[thm]{Theorem-Conjecture}
\numberwithin{equation}{section}

\theoremstyle{definition}

\allowdisplaybreaks

\newcommand{\f}{\Bbb F}

\begin{document}

\title[a Type of Permutation Rational Functions over Finite Fields]{On a Type of Permutation Rational Functions over Finite Fields}

\author[Xiang-dong Hou]{Xiang-dong Hou}
\address{Department of Mathematics and Statistics,
University of South Florida, Tampa, FL 33620}
\email{xhou@usf.edu}

\author[Christopher Sze]{Christopher Sze}
\address{Department of Mathematics and Statistics,
University of South Florida, Tampa, FL 33620}
\email{csze@mail.usf.edu}

\keywords{absolute irreducibility, finite field, Hasse-Weil bound, permutation, rational function}

\subjclass[2010]{11R58, 11T06, 12E12, 14H05}

\begin{abstract}
Let $p$ be a prime and $n$ be a positive integer. Let $f_b(X)=X+(X^p-X+b)^{-1}$, where $b\in\f_{p^n}$ is such that $\text{Tr}_{p^n/p}(b)\ne 0$. In 2008, Yuan et al. \cite{Yuan-Ding-Wang-Pieprzyk-FFA-2008} showed that for $p=2,3$, $f_b$ permutes $\f_{p^n}$ for all $n\ge 1$. Using the Hasse-Weil bound, we show that when $p>3$ and $n\ge 5$, $f$ does not permute $\f_{p^n}$. For $p>3$ and $n=2$, we prove that $f_b$ permutes $\f_{p^2}$ if and only if $\text{Tr}_{p^2/p}(b)=\pm 1$. We conjecture that for $p>3$ and $n=3,4$, $f_b$ does not permute $\f_{p^n}$.
\end{abstract}

\maketitle

\section{Introduction}

Let $p$ be a prime, $n$ be a positive integer, and $\f_{p^n}$ be the finite field with $p^n$ elements. Let $b\in\f_{p^n}$ be such that $\text{Tr}_{p^n/p}(b)\ne 0$. There is considerable interest in permutations of $\f_{p^n}$ of the form $X+(X^p-X+b)^s$, where $s\in\Bbb Z$; see \cite{Helleseth-Zinoviev-FFA-2003, Hou-FFA-2015a, Li-Wang-Li-Zeng-FFA-2018, Tu-Zeng-Jiang-FFA-2015, Yuan-Ding-FFA-2007, Yuan-Ding-Wang-Pieprzyk-FFA-2008, Zeng-Zhu-Hu-FFA-2010, Zha-Hu-FFA-2016, Zheng-Yuan-Yu-FFA-2019} and the references therein. The case $s=-1$ turns out to be quite interesting. Let 
\begin{equation}\label{1.1}
f_b(X)=X+\frac 1{X^p-X+b},
\end{equation}
In 2008, Yuan et al. \cite{Yuan-Ding-Wang-Pieprzyk-FFA-2008} showed that for $p=2,3$, $f_b$ permutes $\f_{p^n}$ for all $n\ge 1$. The results were rediscovered in \cite{Ferraguti-Micheli-arXiv1805.03097} and \cite{hou-ppt} in a different context. Naturally, one would like to know if the statement is true for $p>3$. We give a negative answer to the question by proving the following theorem.

\begin{thm}\label{T1.1}
Let $p>3$ be a prime and $n\ge 5$ be an integer. Let $b\in\f_{p^n}$ be such that $\text{\rm Tr}_{p^n/p}(b)\ne 0$. Then the rational function $f_b$ in \eqref{1.1} does not permute $\f_{p^n}$.
\end{thm}

In Section~2, we prove Theorem~\ref{T1.1} using the Aubry-Perret version \cite{Aubry-Perret-1993} of the Hasse-Weil bound \cite{Weil-1948}. To apply the Hasse-Weil bound, we need to show that a certain polynomial in $\f_{p^n}[X,Y]$ is absolutely irreducible. The same approach has been used for different questions \cite{Blokhuis-Cao-Chou-Hou-JNT-2018, Hou-FFA-2018}. In Section~3, we show that for $p>3$ and $n=2$, $f_b$ permutes $\f_{p^2}$ if and only if $\text{Tr}_{p^2/p}(b)=\pm1$. We use a novel method for solving polynomial equations in $x,x^p,y,y^p$ in $\f_{p^2}$. We conclude the paper with a conjecture about the unsettled cases ($n=3,4$) and a few brief remarks.

\section{Proof of Theorem~\ref{T1.1}}

For $p^n=5^5,7^5,11^5,13^5$, Theorem~\ref{T1.1} has been verified using a computer. Therefore, we assume that either $n\ge 6$ or $n=5$ and $p\ge 17$.

It suffices to show that there exist $x\in\f_{p^n}$ and $y\in\f_{p^n}^*$ such that $f_b(x+y)-f_b(x)=0$. We have
\[
f_b(x+y)-f_b(x)=\frac{y(z^2+(y^p-y)z+1-y^{p-1})}{(x^p-x+b)((x+y)^p-(x+y)+b)},
\]
where $z=x^p-x+b$. Let
\begin{equation}\label{2.1a}
F(X,Y)=(X^p-X+b)^2+(Y^p-Y)(X^p-X+b)+1-Y^{p-1}\in\f_{p^n}[X,Y].
\end{equation}
It suffices to show that $F$ has a zero $(x,y)\in\f_{p^n}^2$ with $y\ne 0$. Let
\begin{equation}\label{2.1.0}
V_{\f_{p^n}^2}(F)=\{(x,y)\in\f_{p^n}^2:F(x,y)=0\}.
\end{equation}
By Lemma~\ref{L} below, $F(X,Y)$ is absolutely irreducible, i.e., irreducible over the algebraic closure $\overline\f_p$ of $\f_p$. Then by the Aubry-Perret version of the Hasse-Weil bound \cite{Aubry-Perret-1993}, we have
\begin{equation}\label{2.2a}
|V_{\f_{p^n}^2}(F)|\ge p^n+1-(2p-1)(2p-2)p^{n/2}-2,
\end{equation}
where $2$ is the number of zeros of $F$ at infinity. Thus
\[
|V_{\f_{p^n}^2}(F)|\ge p^n-4p^{2+n/2}-1=p^{2+n/2}(p^{n/2-2}-4)-1.
\]
When $n\ge 6$, $p^{n/2-2}-4\ge p-4\ge 1$; when $n=5$ and $p\ge 17$, $p^{n/2-2}-4\ge\sqrt{17}-4>0.1$. Hence
\[
|V_{\f_{p^n}^2}(F)|> 0.1 p^{2+n/2}-1.
\]
On the other hand, the number of zeros of $F(X,0)=(X^p-X+b)^2+1$ in $\f_{p^n}$ is
\[
\le 2p<0.1p^{2+n/2}-1<|V_{\f_{p^n}^2}(F)|.
\]
Therefore $F(X,Y)$ has a zero $(x,y)\in\f_{p^n}^2$ with $y\ne 0$.

\begin{lem}\label{L} 
The polynomial $F(X,Y)$ in \eqref{2.1a} is absolutely irreducible. 
\end{lem}

\begin{proof}
View $F(X,Y)$ as a polynomial in $X$ over $\overline\f_p[Y]$; it is clearly primitive (the gcd of the coefficients is $1$). Hence it suffices to show that $F(X,Y)$ is irreducible over $\overline\f_p(Y)$. Let $x$ be a root of $F(X,Y)$ and let $z=x^p-x+b$. It suffices to show that $[\overline\f_p(x,Y):\overline\f_p(Y)]=2p$, and to this end, it suffices to show that $[\overline\f_p(z,Y):\overline\f_p(Y)]=2$ and $[\overline\f_p(x,Y):\overline\f_p(z,Y)]=p$.

\[
\beginpicture
\setcoordinatesystem units <1em,1em> point at 0 0

\setlinear
\plot 0 1  0 3 /
\plot 0 5  0 7 /

\put {$\overline\f_p(Y)$} at 0 0
\put {$\overline\f_p(z,Y)$} at 0 4
\put {$\overline\f_p(x,Y)$} at 0 8
\put {$\scriptstyle 2$} [l] at 0.4 2
\put {$\scriptstyle p$} [l] at 0.4 6

\endpicture
\]
\smallskip

$1^\circ$ We claim that $[\overline\f_p(z,Y):\overline\f_p(Y)]=2$. We have
\[
z^2+(Y^p-Y)z+1-Y^{p-1}=0.
\]
Hence
\[
z=\frac 12(Y-Y^p+\Delta),
\]
where 
\[
\Delta^2=(Y^p-Y)^2-4(1-Y^{p-1})=(Y^{p-1}-1)(Y^2(Y^{p-1}-1)+4).
\]
Obviously, $\Delta\notin\overline\f_p(Y)$, so $z$ is of degree $2$ over $\overline\f_p(Y)$.

\medskip
$2^\circ$ We claim that $[\overline\f_p(x,Y):\overline\f_p(z,Y)]=p$. Since $x$ is a root of $X^p-X-z+b$, the claim is equivalent to saying that $X^p-X-z+b$ is irreducible over $\overline\f_p(z,Y)$. Assume to the contrary that $X^p-X-z+b$ is reducible over $\overline\f_p(z,Y)=\overline\f_p(\Delta,Y)$. Then all its roots are in $\overline\f_p(\Delta,Y)$, so $x=A\Delta+B$ for some $A,B\in\overline\f_p(Y)$ with $A\ne 0$. Then 
\begin{align*}
0\,&=(A\Delta+B)^p-(A\Delta+B)-z+b\cr
&=A^p\Delta^p+B^p-A\Delta-B-\frac 12(Y-Y^p+\Delta)+b\cr
&=\Delta\Bigl(A^p\Delta^{p-1}-A-\frac 12\Bigr)+B^p-B-\frac 12(Y-Y^p)+b\cr
&=\frac 12\Delta\bigl(A^p((Y^{p-1}-1)(Y^2(Y^{p-1}-1)+4))^{(p-1)/2}-2A-1\bigr)
\cr
&\phantom{=\,} -\frac 12(Y-Y^p)+B^p-B+b.
\end{align*}
Hence
\begin{equation}\label{3.1}
A^p((Y^{p-1}-1)(Y^2(Y^{p-1}-1)+4))^{(p-1)/2}-2A-1=0.
\end{equation}
We claim that $(Y^{p-1}-1)(Y^2(Y^{p-1}-1)+4)$ has no multiple zeros. For this claim, it suffices to show that $Y^2(Y^{p-1}-1)+4=Y^{p+1}-Y^2+4$ has no multiple zeros. In fact,
\begin{align*}
&\text{gcd}(Y^{p+1}-Y^2+4,(Y^{p+1}-Y^2+4)')=\text{gcd}(Y^{p+1}-Y^2+4,Y^p-2Y)\cr
=\,&\text{gcd}(Y^{p+1}-Y^2+4,Y^{p-1}-2)=\text{gcd}(Y^2+4,Y^{p-1}-2)\cr
=\,&\text{gcd}(Y^2+4,(-4)^{(p-1)/2}-2)=\text{gcd}(Y^2+4,(-1)^{(p-1)/2}-2)=1.
\end{align*}

If $A\in\overline\f_p[Y]$, then the left side of \eqref{3.1} is a polynomial in $Y$ of degree $>0$, which is a contradiction. Now assume that $A$ has a pole $P\ne\infty$, and let $\nu_P$ denote the valuation of $\overline\f_p(Y)$ at $P$. Then 
\begin{align*}
&\nu_P\bigl(A^p((Y^{p-1}-1)(Y^2(Y^{p-1}-1)+4))^{(p-1)/2}\bigr)\cr
=\,&p\,\nu_P(A)+\frac{p-1}2\nu_P\bigl((Y^{p-1}-1)(Y^2(Y^{p-1}-1)+4)\bigr)\cr
\le\,& p\,\nu_P(A)+\frac{p-1}2<\nu_P(A),
\end{align*}
which is a contradiction to \eqref{3.1}.
\end{proof}

\section{The Case $n=2$}

\begin{thm}\label{T3.1}
Assume that $p\ge 3$ and $b\in\f_{p^2}$ is such that $\text{\rm Tr}_{p^2/p}(b)\ne 0$. Then $f_b(X)$ permutes $\f_{p^2}$ if and only if $\text{\rm Tr}_{p^2/p}(b)=\pm 1$.
\end{thm}

\begin{proof}
Let $\tau=\text{Tr}_{p^2/p}(b)$ and let $F(X,Y)$ be the polynomial in \eqref{2.1a}. Recall that $f_b$ does not permute $\f_{p^n}$ if and only if $F(X,Y)$ has a zero $(x,y)\in\f_{p^n}$ with $y\ne 0$. 

($\Leftarrow$) Assume to the contrary that $f_b(X)$ does not permute $\f_{p^2}$. Then $F(x,y)=0$ for some $(x,y)\in\f_{p^2}^2$ with $y\ne 0$, i.e., 
\begin{equation}\label{4.2}
z^2+(y^p-y)z+1-y^{p-1}=0,
\end{equation}
where $z=x^p-x+b$.
Raising \eqref{4.2} to the power of $p$ gives
\begin{equation}\label{4.3}
(\tau-z)^2+(y-y^p)(\tau-z)+1-y^{1-p}=0.
\end{equation}
Subtracting \eqref{4.3} from \eqref{4.2} yields
\begin{equation}\label{3.3a}
\tau(2z-\tau)-(y-y^p)\tau-y^{p-1}+y^{1-p}=0.
\end{equation}
Hence 
\begin{equation}\label{3.4a}
z=\frac 12(\tau+y-y^p+\tau^{-1}y^{p-1}-\tau^{-1}y^{1-p}).
\end{equation}
Making this substitution in \eqref{4.2} and using the fact $\tau^2=1$, we have
\[
-\frac 14y^{-2-2p}(y^2+y^{2p}-y^{1+p}-y^{2+p}+y^{1+2p})(-y^2-y^{2p}+y^{1+p}-y^{2+p}+y^{1+2p})=0.
\]
Without loss of generality, assume that 
\begin{equation}\label{4.4}
y^2+y^{2p}-y^{1+p}-y^{2+p}+y^{1+2p}=0. 
\end{equation}
Then $2y^{1+p}y=y^2+y^{2p}-y^{1+p}+y^{2+p}+y^{1+2p}\in\f_p$, and hence $y\in\f_p$. Now \eqref{4.4} becomes $y^2=0$, which is a contradiction.

\medskip
($\Rightarrow$)
Assume to the contrary that $\tau\ne 0,\pm 1$. We show that there exists $(x,y)\in\f_{p^2}^2$ with $y\ne 0$ such that $F(x,y)=0$. 
We assume that $p>2^{10}$. (For $p<2^{10}$, the claim has been verified by computer.)

Recall that 
\begin{equation}\label{2.1}
F(X,Y)=Z^2+(Y^p-Y)Z+1-Y^{p-1},
\end{equation}
where $Z=X^p-X+b$. Assume temporarily that $(z,y)\in\f_{q^2}^2$ is such that $y\ne0$, $\text{Tr}_{p^2/p}(z)=\text{Tr}_{p^2/p}(b)=\tau$ and
\begin{equation}\label{2.2}
z^2+(y^p-y)z+1-y^{p-1}=0.
\end{equation}
Then \eqref{4.3} -- \eqref{3.4a} still holds. Making the substitution \eqref{3.4a} in \eqref{2.2} yields
\begin{align}\label{2.5}
\frac 14\tau^{-2}y^{-2-2p}&\bigl(y^4+y^{4p}-2\tau^2y^{3+p}+(-2+4\tau^2+\tau^4)y^{2+2p}\\
&-\tau^2y^{4+2p}-2\tau^2y^{1+3p}+2\tau^2y^{3+3p}-\tau^2y^{2+4p}\bigr)=0.\nonumber
\end{align}
We write \eqref{2.5} as
\begin{equation}\label{2.6}
G(y,y^p)=0,
\end{equation}
where 
\begin{align}\label{2.7}
G(X,Y)=\,&X^4+Y^4-2\tau^2X^3Y+(-2+4\tau^2+\tau^4)X^2Y^2\\
&-\tau^2 X^4Y^2-2\tau^2 XY^3+2\tau^2 X^3Y^3-\tau^2 X^2Y^4\in\f_p[X,Y].\nonumber
\end{align}

We now remove the temporary assumption about $(z,y)$. We will show that \eqref{2.6} has a solution $y\in\f_{p^2}\setminus\f_p$. Once this is done, the proof of the theorem is completed as follows: Let $z$ be given by \eqref{3.4a}. Since $\text{Tr}_{p^2/p}(z)=\tau=\text{Tr}_{p^2/p}(b)$, there exists $x\in\f_{p^2}$ such that $z=x^p-x+b$. It follows that in \eqref{2.1}, $F(x,y)=0$.

Let 
\[
V_{\f_p^2}(G)=\{(x,y)\in\f_p^2:G(x,y)=0\}.
\]
By Lemma~\ref{L2.2}, $G(X,Y)$ is absolutely irreducible. Then by the Aubry-Perret version of the Hasse-Weil bound, 
\[
|V_{\f_p^2}(G)|\le p+1+(6-1)(6-2)p^{1/2}-3,
\]
where $3$ is the number of zeros on $G$ at infinity. Thus
\[
|V_{\f_p^2}(G)|\le p+20p^{1/2}-2.
\]
$G(X,Y)$ is a symmetric polynomial in $X$ and $Y$, which can be written as
\begin{equation}\label{2.8}
G(X,Y)=H(X+Y,XY),
\end{equation}
where 
\begin{equation}\label{2.9}
H(X,Y)=X^4-(4+2\tau^2)X^2Y+(8\tau^2+\tau^4)Y^2-\tau^2X^2Y^2+4\tau^2Y^3.
\end{equation}
Since $G(X,Y)$ is absolutely irreducible, so is $H(X,Y)$. Again, by the Aubry-Perret version of the Hasse-Weil bound, 
\[
|V_{\f_p^2}(H)|\ge p+1-(4-1)(4-2)p^{1/2}-3,
\]
where $3$ is the number of zeros on $H$ at infinity. Thus
\[
|V_{\f_p^2}(H)|\ge p-6p^{1/2}-2.
\]
Note that $G(X,X)=\tau^4X^4$. Hence $V_{\f_p^2}(G)\cap\{(x,x):x\in\f_p\}=\{(0,0)\}$. The map
\[
\begin{array}{cccc}
\phi:& V_{\f_p^2}(G)\setminus\{(0,0)\}&\longrightarrow & V_{\f_p^2}(H)\cr
&(x,y)&\longmapsto &(x+y,xy)
\end{array}
\]
is 2-to-1. Hence
\[
|\phi(V_{\f_p^2}(G))|=1+\frac 12(|\phi(V_{\f_p^2}(G))|-1)\le 1+\frac 12(p+20p^{1/2}-3)=\frac 12(p+20p^{1/2}-1).
\]
Since $p>2^{10}$,
\[
\frac 12(p+20p^{1/2}-1)< p-6p^{1/2}-2\le |V_{\f_p^2}(H)|.
\]
Therefore, there exists $(u,v)\in V_{\f_p^2}(H)\setminus\phi(V_{\f_p^2}(G))$. It follows that the quadratic polynomial $X^2-uX+v$ is irreducible over $\f_p$. Let $y\in\f_{p^2}\setminus\f_p$ be a root of $X^2-uX+v$. Then $y+y^p=u$ and $yy^p=v$. Hence 
\[
G(y,y^p)=H(u,v)=0.
\]
\end{proof}

\begin{lem}\label{L2.2}
Assume that $\tau\ne 0,\pm1$. The polynomial $G(X,Y)$ in \eqref{2.7} is absolutely irreducible.
\end{lem}

\begin{proof}
Assume to the contrary that $G(X,Y)$ is not absolutely irreducible. Let $t=\tau^2$ and consider the homogenization $\overline G(X,Y,Z)$ of $G(X,Y)$:
\[
\overline G(X,Y,Z)=A(X,Y)Z^2-tX^2Y^2(X-Y)^2,
\]
where 
\[
A(X,Y)=X^4-2tX^3Y+(-2+4t+t^2)X^2Y^2-2tXY^3+Y^4.
\]
Since $\overline G(X,Y,Z)$ is reducible over $\overline\f_p$, so is $\overline G(X,1,Z)$. We have
\[
\overline G(X,1,Z)=A(X,1)Z^2-tX^2(X-1)^2,
\]
where 
\[
A(X,1)=X^4-2tX^3+(-2+4t+t^2)X^2-2tX+1.
\]
Since $\text{gcd}(A(X,1),tX^2(X-1)^2)=1$, $\overline G(X,1,Z)$ is a product of two linear polynomials in $Z$ over $\overline\f_p[X]$. It follows that $A(X,1)$ is a square in $\overline\f_p[X]$, that is,
\[
X^4-2tX^3+(-2+4t+t^2)X^2-2tX+1-(X^2+\alpha X+\beta)^2=0
\]
for some $\alpha,\beta\in\overline\f_p$. Comparing the coefficients gives
\begin{gather}
\label{2.10} 
1-\beta^2=0,\\
\label{2.11}
-2(\alpha \beta+t)=0,\\
\label{2.12}
-2-\alpha^2-2\beta+4t+t^2=0,\\
\label{2.13}
-2(\alpha+t)=0.
\end{gather}
By \eqref{2.13}, we have $\alpha=-t$. Then \eqref{2.11} gives $\beta=1$. Now \eqref{2.12} becomes $4(-1+t)=0$, i.e., $t=1$, which is a contradiction.
\end{proof}

\begin{rmk}\label{R4.3}\rm
Theorem~\ref{T3.1} and its proof hold verbatim if we replace $p$ by $q=p^n$ (still assuming $p\ge 3$).
\end{rmk}

\section{Remarks and a Conjecture}

Recall that
\[
f_b(X)=X+\frac 1{X^p-X+b},
\]
where $b\in\f_{p^n}$ is such that $\text{Tr}_{p^n/p}(b)\ne 0$. If $b_1\in\f_{p^n}$ is such that $\text{Tr}_{p^n/p}(b_1)=\epsilon\text{Tr}_{p^n/p}(b)$, where $\epsilon=\pm 1$, then $b_1=\epsilon b+c^p-c$ for some $c\in\f_{p^n}$. Therefore
\begin{align*}
f_b(\epsilon X+c)\,&=\epsilon X+c+\frac 1{\epsilon(X^p-X)+\epsilon(c^p-c)+b}\cr
&=\epsilon\Bigl(X+\frac 1{X^p-X+c^p-c+\epsilon b}\Bigr)+c\cr
&=\epsilon\Bigl(X+\frac 1{X^p-X+b_1}\Bigr)+c=\epsilon f_{b_1}(X)+c.
\end{align*}
Hence $f_b(X)$ permutes $\f_{p^n}$ if and only if $f_{b_1}(X)$ does. In particular, when studying the permutation properties of $f_b(X)$ over $\f_{p^n}$ for $p\ge 5$ and $n\le 4$, we may assume that $b\in\{1,2,\dots,(p-1)/2\}$ since for such $b$, $\text{Tr}_{p^n/p}(b)$ takes all values of $\f_p^*$ modulo a $\pm$ sign.

The family $f_b$ contains three parameters: $p,n,b$. When $p\le 3$ or $n\ge 5$, all triples $(p,n,b)$ so that $f_b$ permutes $\f_{p^n}$ have been determined. The remaining cases are $n=3,4$ with $p\ge 5$. Extensive computer search indicates that $f_b$ never permutes $\f_{p^n}$ in these two cases, which leads us to the following conjecture.

\begin{conj}\label{C4.1}
Let $p\ge 5$, $n=3$ or $4$, and $b\in \f_{p^n}$ be such that $\text{\rm Tr}_{p^n/p}(b)\ne 0$. Then $f_b$ does not permute $\f_{p^n}$.
\end{conj}

As we have seen above, whether $f_b$ permutes $\f_{p^n}$ depends on $\text{\rm Tr}_{p^n/p}(b)^2$ rather than $b$. For $n=4$, we may assume that $0\ne b\in\f_p$ and $\text{\rm Tr}_{p^2/p}(b)^2=1$. (If $\text{\rm Tr}_{p^2/p}(b)^2\ne 1$, then by Theorem~\ref{T3.1}, $f_b$ does not permute $\f_{p^2}$ and hence does not permute $\f_{p^4}$.) Therefore Conjecture~\ref{C4.1} can be stated in a more precise form:

\medskip
\noindent{\bf Conjecture}~$\boldsymbol{4.1'.}$
{\it Let $p\ge 5$. For $n=3$ and $b\in\{1,2,\dots,(p-1)/2\}$, $f_b$ does not permute $\f_{p^3}$. For $n=4$, $f_{1/2}$ does not permute $\f_{p^4}$.



\end{document}